\documentclass[10pt, twocolumn, conference]{IEEEtran}

\makeatletter
\def\ps@headings{%
\def\@oddhead{\mbox{}\scriptsize\rightmark \hfil \thepage}%
\def\@evenhead{\scriptsize\thepage \hfil \leftmark\mbox{}}%
\def\@oddfoot{}%
\def\@evenfoot{}}
\makeatother
\usepackage{amsmath, mathrsfs,amsfonts}
\usepackage{amsfonts}
\usepackage{amssymb}
\usepackage{acronym}
\usepackage{graphicx}
\usepackage{cite}
\usepackage{algorithmic}
\usepackage{algorithm2e}
\usepackage{array}
\usepackage{psfrag}
\usepackage{subfigure}

\def\ignore#1{}

\newtheorem{theorem}{Theorem}
\newtheorem{lemma}{Lemma}
\newtheorem{corollary}{Corollary}

               {\begin{list}{}{\leftmargin#1\rightmargin#2\topsep#3}\item{}}%
               {\end{list}}

\IEEEoverridecommandlockouts
\begin{document}
\title{Robustness of Interdependent Networks: The case of communication networks and the power grid}

\author{ Marzieh Parandehgheibi and Eytan Modiano\\
\normalsize{Laboratory for Information and Decision Systems}\\
\normalsize{Massachusetts Institute of Technology}\\
\normalsize{Cambridge, MA 02139}}

\maketitle \thispagestyle{headings}

\begin{abstract}
In this paper, we study the robustness of interdependent networks, in which the state of one network depends on the state of the other network and vice versa. In particular, we focus on the interdependency between the power grid and communication networks, where the grid depends on communications for its control, and the communication network depends on the grid for power.  A real-world example is the Italian blackout of 2003, when a small failure in the power grid cascaded between the two networks and led to a massive blackout. In this paper, we study the minimum number of node failures needed to cause total blackout (i.e., all nodes in both networks to fail).  In the case of unidirectional interdependency between the networks we show that the problem is NP-hard, and develop heuristics to find a near-optimal solution. On the other hand, we show that in the case of bidirectional interdependency this problem can be solved in polynomial time.  We believe that this new interdependency model gives rise to important, yet unexplored, robust network design problems for interdependent networked infrastructures.
\end{abstract}

\section{Introduction}
Nowadays infrastructure networks are interdependent in such a way that the failure of an element in one network may cascade to another network and cause the failure of dependent elements. This failure procedure can cascade multiple times between two or more interdependent networks and result in catastrophic widespread failures. A particular example is the interdependency between the power grid and communication networks; where the 2003 Italian blackout, affecting the lives of fifty-five million people, was the result of such interdependency \cite{Rosato}. While reliability in networks has been studied extensively, the focus of most efforts has been on single networks in isolation. However, coupled networks exhibit very different behavior than isolated networks due to the cascading failure effects. In this paper, we develop a simple model for the interdependency between the networks, and analyze their robustness.

In order to enhance the reliability of the power grid, it is essential to have a smarter grid that detects autonomously faults and allows self-healing \cite{MassoudAmin}. The control of today's power grid relies on a supervisory control and data acquisition (SCADA) system. This SCADA system does not have the capability to collect real-time data from the entire grid and analyze it efficiently to make control decisions. It is expected that future smart grid will need a wide-area measurement system (WAMS) which uses new technologies such as Phasor Measurement Units (PMUs) and Sensor and Actuator Networks (SANETs) to collect synchronous data, send it to control centers for making online control decisions, and apply changes to the grid using devices such as actuators. This complex control network increases the dependency of the power grid on communication networks and gives rise to the necessity of analyzing and designing robust cyber-physical networks \cite{VincentPoor}.

There have been a few attempts at introducing and modeling interdependent infrastructure networks. Rinaldi \textit{et al.} describe interdependencies among major infrastructure networks including the power grid and communication networks. In particular, they show that all infrastructure networks depend on information delivered by the communication network for monitoring and controlling their subsystems (i.e. SCADA). At the same time, they show that communication networks depend on the physical output of other networks; for example, power from the electric grid for switches or water from the water infrastructure for cooling systems. These interdependencies are referred to as cyber and physical, respectively \cite{Rinaldi}.

Later, Rosato \textit{et al.} studied the Italian blackout of 2003 which affected the lives of fifty-five million people. They explain that this blackout was the result of a cascade of failures between the power grid and the communication network. In fact, due to failures in the power grid, some of the switches in the communication network lost their power and failed. Subsequently, due to failures in communication network, some of the substations in the power grid lost their control, and failed. They also demonstrate that failures inside the power grid lead to failures in the communication network using simulations on the real data of the Italian network \cite{Rosato}. Although the papers by Rinaldi and Rosato explain the concept of interdependency between networks, neither presents a mathematical model for investigating the behavior of interdependent networks. Moreover, there has been some work on the interdependency within layered communication infrastructure, e.g. between fiber optical networks and IP networks \cite{Kayi,Parandehgheibi}; but, they did not consider interdependency between different infrastructures.

In 2009, Buldyrev \textit{et al.} presented a model for analyzing the robustness of interdependent random networks. They generate two disjoint random graphs and define a one-to-one dependency between every pair of nodes in these graphs. Using techniques from percolation theory, they show that random failures cascade through the graphs, and investigate the existence of a giant component \cite{Nature}. Later, Parshani \textit{et al.} showed that reducing the dependency between the networks makes them more robust to random failures \cite{Parshani}. A discussion of follow-up works on Buldyrev's model can be found in \cite{Gao}.

However, real networks such as the power grid and the communication infrastructure are not random. These networks have known topologies; thus, it is essential to understand the impact of failure cascades in such networks in order to design robust interdependent infrastructures. To the best of our knowledge, this paper is the first attempt to model the interdependency between the networks with known topologies. The focus of this work is on the interdependency between power grid and communication networks; however, it can be extended to other networks with cyber-physical interdependency.

In Section \ref{Model_sec}, we present the network model and define a new metric for assessing the robustness of networks. In Section \ref{STAR}, we study the interdependency between networks with star topologies. In Section \ref{discussion}, we evaluate the robustness of a real network and discuss future research directions, and finally, conclude in Section \ref{conclusion}.

\section{Cyber-Physical Interdependency Model}\label{Model_sec}
We start with a simple model for the power grid and the supporting Control and Communication Network (CCN). The power grid consists of generators $G$ and substations $S$ which are connected with power lines. Similarly, the CCN consists of control centers $C$ and routers $R$ which are connected with communication lines. Every router receives power from at least one substation and every substation sends data and receives control signals from at least one router. Figure \ref{Model} shows the interdependency model between the networks.

\begin{figure}[ht]
\centering
\includegraphics[scale=0.33, angle=-90]{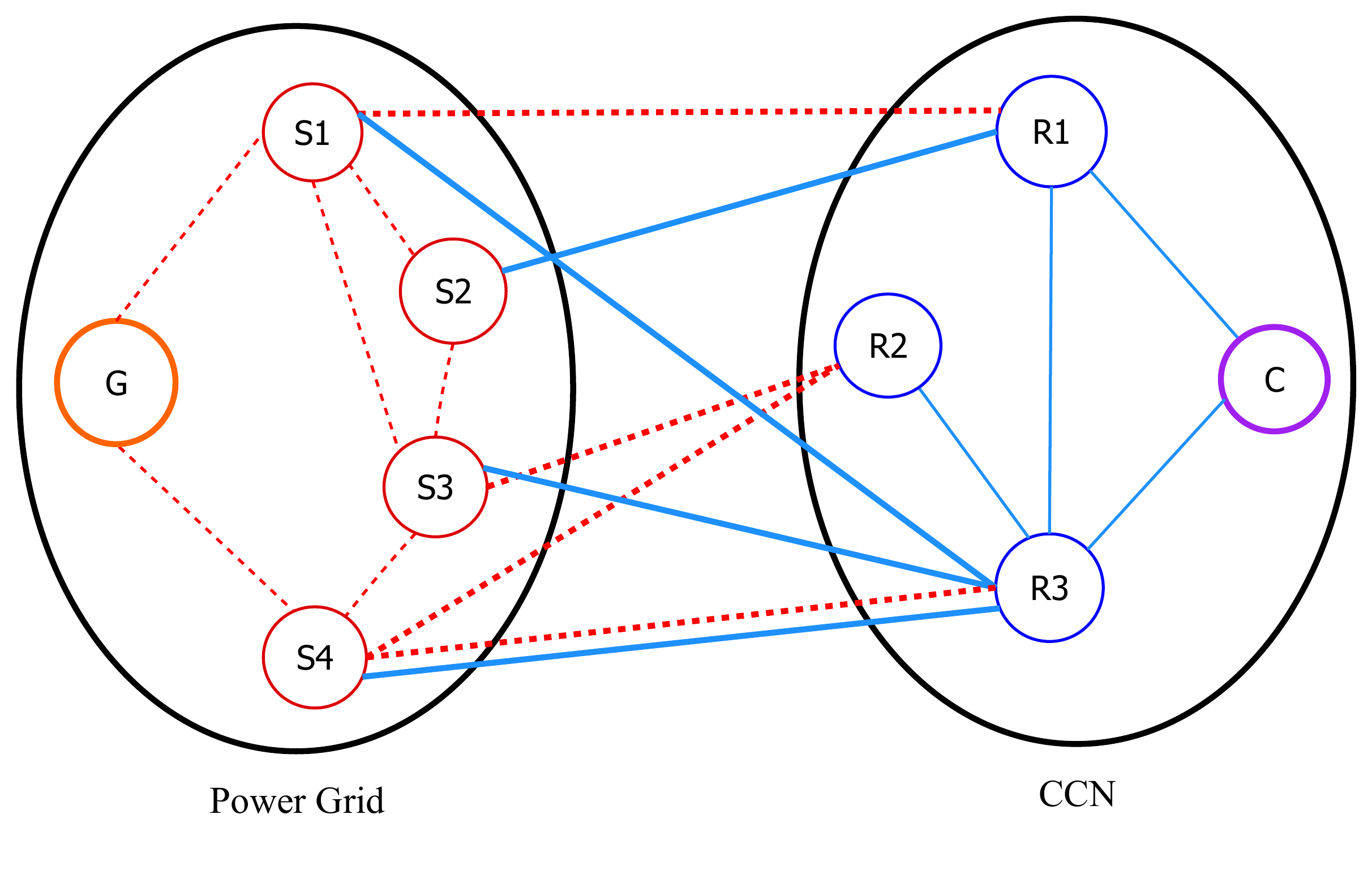}\vspace{-0.3cm}
\caption{Cyber-Physical Interdependency Model - dotted lines represent power lines and solid lines represent communication lines}
\label{Model}\vspace{-0.2cm}
\end{figure}

In this model, we say that a substation \textit{operates} if it has a path to a generator, i.e. receives power, and it is also connected to a router, i.e. sends data and receives control signals. Similarly, we say that a router \textit{operates} if it has a path to a control center, i.e. sends data and receives control signals and it is also connected to a substation, i.e. receives power. Thus, a failure in the power grid may cause a failure in the CCN and vice versa. For example, future substations will be equipped with FACTS \footnote{Flexible Alternating Current Transmission System} which control the reactance of the power lines attached to a substation and prevent those lines from overloading and failure. These FACTS receive their control through the CCN; thus, the substations connected to a router will be controlled and do not experience power overloading. On the other hand, if a substation is not connected to a router (i.e. is not controlled), it may be overloaded and fail. 

We assume that the power generators have internal control, and the control centers have backup generators; thus, they are robust to failures. In addition, in this paper, we do not consider the power flow equations, and assume that the connectivity of a substation to a generator is sufficient for receiving power. We also do not consider the amount of power supply or demand, and only focus on connectivity. Moreover, without loss of generality, we assume that there is only one generator and one control center; this is due to our model's assumption that every substation can be connected to any generator, and every router can be connected to any control centers. Although this model is not fully realistic, it captures the essential properties of interdependent networks.

\subsection{Effect of a Single Failure}\label{Example}
We start with an example demonstrating that a single failure can cascade multiple times within and between the power grid and CCN (Figure \ref{Cascade}). Suppose that initially substation $S_4$ fails (Step 1). As a result, all the edges attached to $S_4$ fail, and router $R_3$ loses its power and fails (Step 2); Consequently, substations $S_1$ and $S_3$ lose their control, and router $R_2$ loses its connection to the control center $C$, and all fail (Step 3). Finally router $R_1$ loses its power, and substation $S_2$ loses its connection to the generator $G$, and both fail (Step 4).

\begin{figure}[h]
\centering
\subfigure[Step1 - $S_4$ fails, initially]
{\label{step1}\includegraphics[scale=0.25, angle=-90]{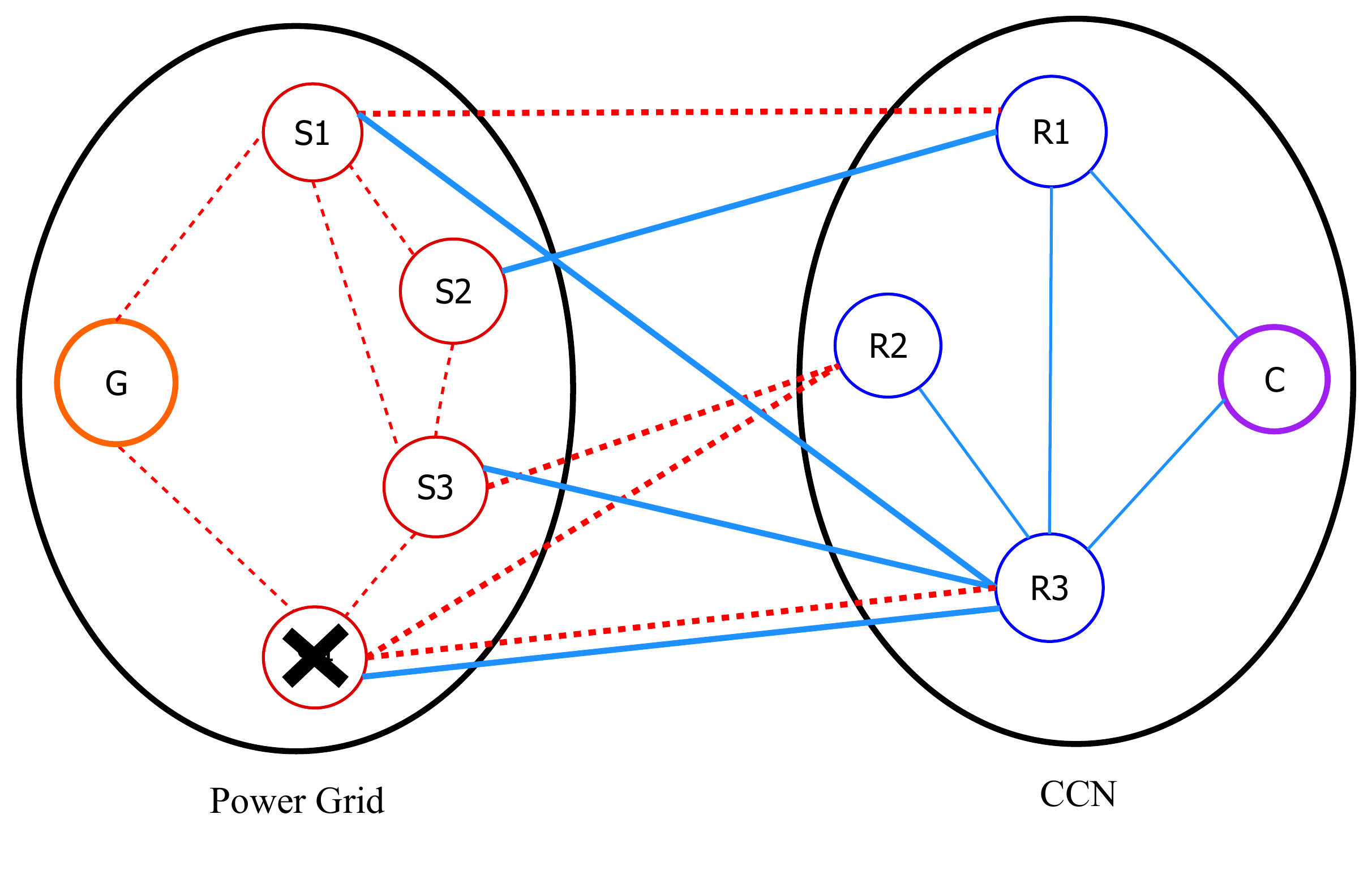}}                
\subfigure[Step2 - $R_3$ fails]
{\label{step2}\includegraphics[scale=0.25, angle=-90]{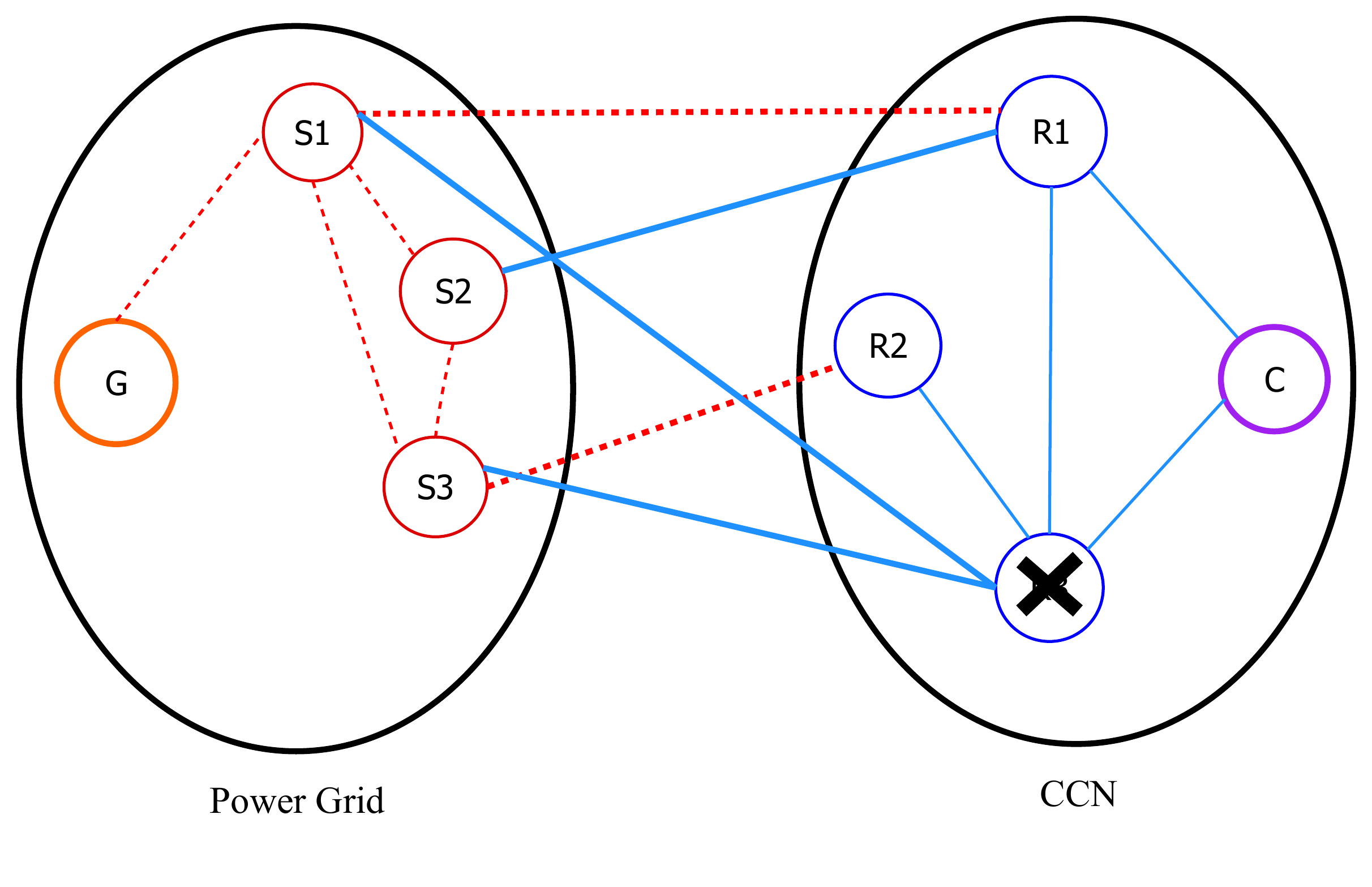}}
\subfigure[Step3 - $S_1$, $S_3$ and $R_2$ fail]
{\label{step3}\includegraphics[scale=0.25, angle=-90]{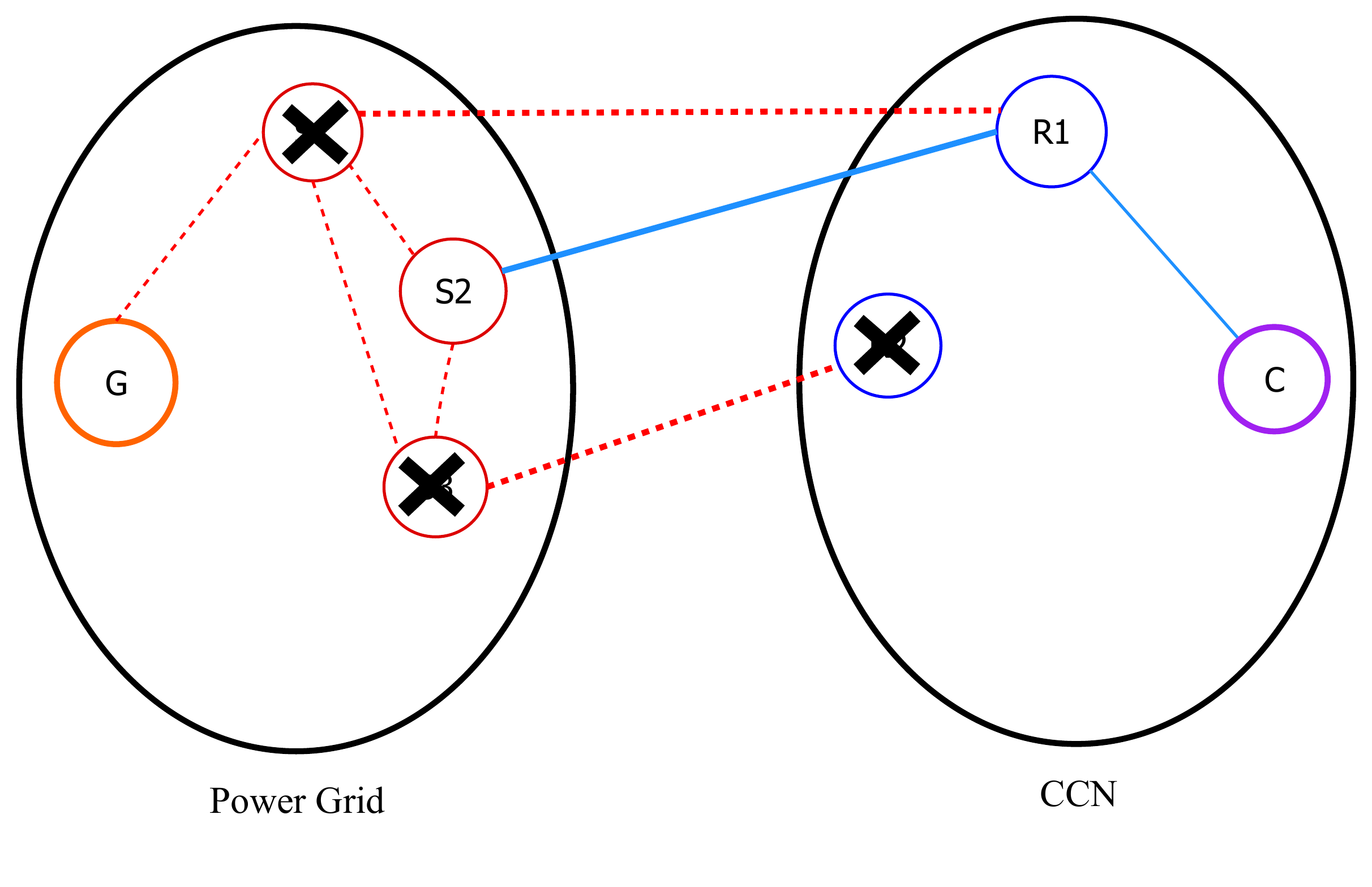}}
\subfigure[Step4 - $R_1$ and $S_2$ fail]
{\label{step4}\includegraphics[scale=0.25, angle=-90]{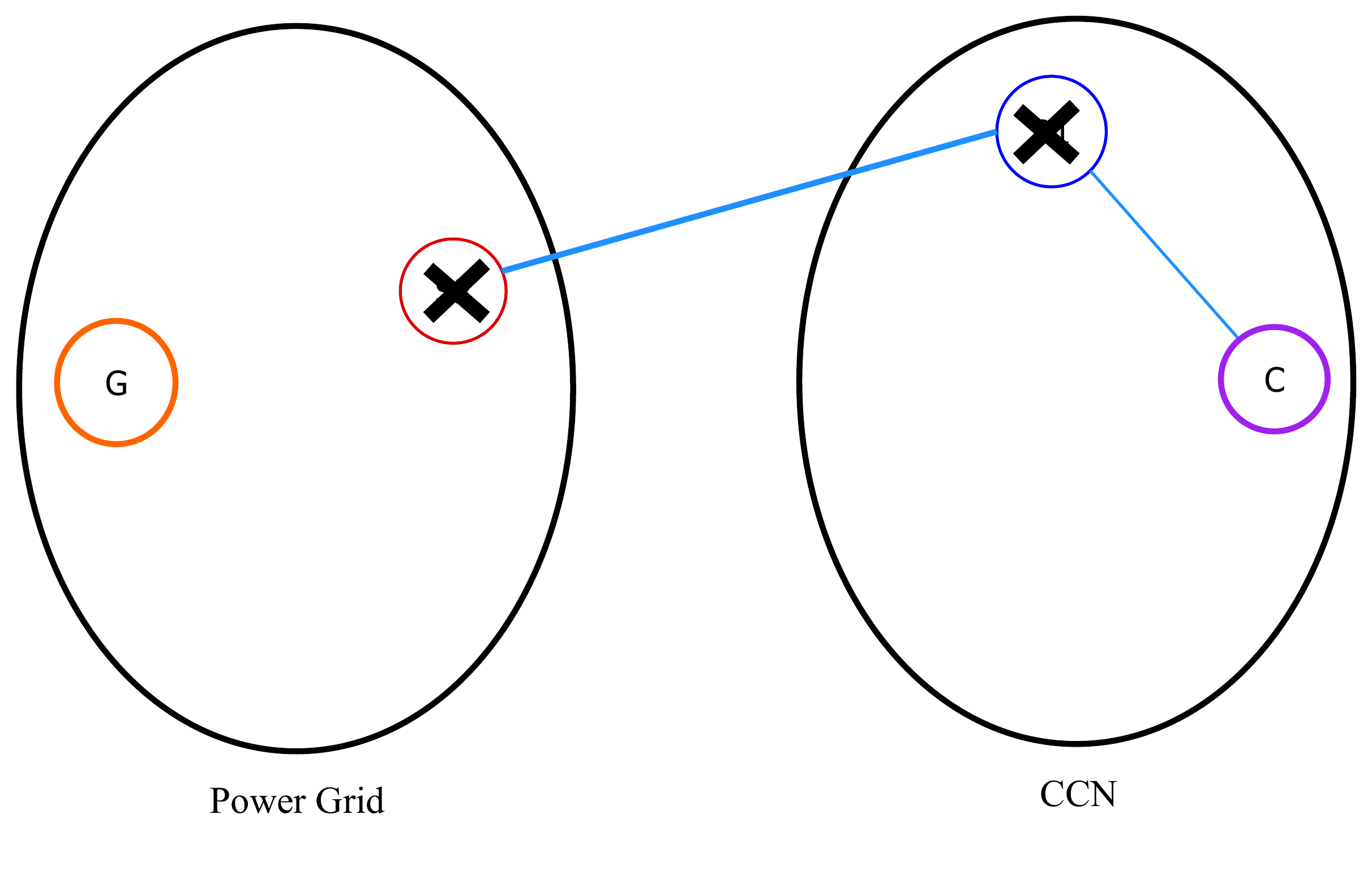}}
\caption{Cascade of a single failure in an interdependent model}
\label{Cascade}
\end{figure}

In contrast, suppose that the power grid is not dependent on the CCN. In this case a substation fails, if and only if it is disconnected from the generator. For example, consider only the power grid from Figure \ref{Model} and, suppose that substation $S_4$ fails 
It can be seen that no other substation will be disconnected from the generator; thus, no further failure occurs. This example indicates that interdependency makes the networks more vulnerable, while it is essential to their operation.


It was seen in the example of Figure \ref{Cascade} that a single failure can lead to the failure of all nodes in both networks. We define the failure of all nodes in both networks as \textit{Total Failure}. This observation leads to an important question. What is the minimum number of nodes whose removals will lead to total failure? We define the minimum total failure removals (MTFR) as a metric that helps to measure the robustness of a network, i.e. the larger the MTFR, the more robust the network. In particular, we consider two types of Node and Edge removals which we refer to as Node-MTFR and Edge-MTFR metrics.

In the case of a single network $G=\{V,E\}$, the smallest set of nodes that can cause a total failure, i.e. disconnect all the nodes in $V$ from source $S$, is the set of nodes directly connected to $S$. 
Consequently, a star topology is the most robust network topology, since all of the nodes in a star are directly connected to the source.


\section{Interdependency between Networks with Star Topologies}\label{STAR}
In order to analyze the interdependency between the power grid and the CCN, we assume that both networks have star topologies. Under this assumption, all of the substations in the power grid are directly connected to the generator; thus no substation's failure can disconnect the other substations from the generator. Similarly, all of the routers in the CCN are directly connected to the control center, and no router's failure can disconnect the other routers from the control center. Therefore, any failure in the system would be \textit{only} due to the interdependency between the networks, i.e. a substation fails if and only if it loses its connection to the CCN, and similarly a router fails if and only if it loses its connection to the power grid. This property gives us the opportunity to study the behavior of interdependent networks, and the effects of interdependency on the robustness of networks. Since both networks have star topologies, the interdependent network under study has a bipartite topology, i.e. no edge connects the nodes inside one network. In the following, we consider two distinct models of unidirectional and bidirectional interdependency. We analyze the networks under each model, and compare their complexity and robustness.

\subsection{Unidirectional Interdependency}\label{Unidir}
Under the unidirectional interdependency model, the edges between the networks are directed, i.e. if a power node $S_i$ provides power for a router $R_j$, the router $R_j$ does not necessarily provide the control signal for the same power node $S_i$ (Figure \ref{Uni1}). 
In the following we show that finding the Node-MTFR is in fact a hitting cycle problem, and prove that it is an NP-complete problem. In order to do so we start with the following lemmas:


\begin{figure}[h]
\centering
\subfigure[Unidirectional Interdependency]
{\label{Uni1}\includegraphics[scale=0.25]{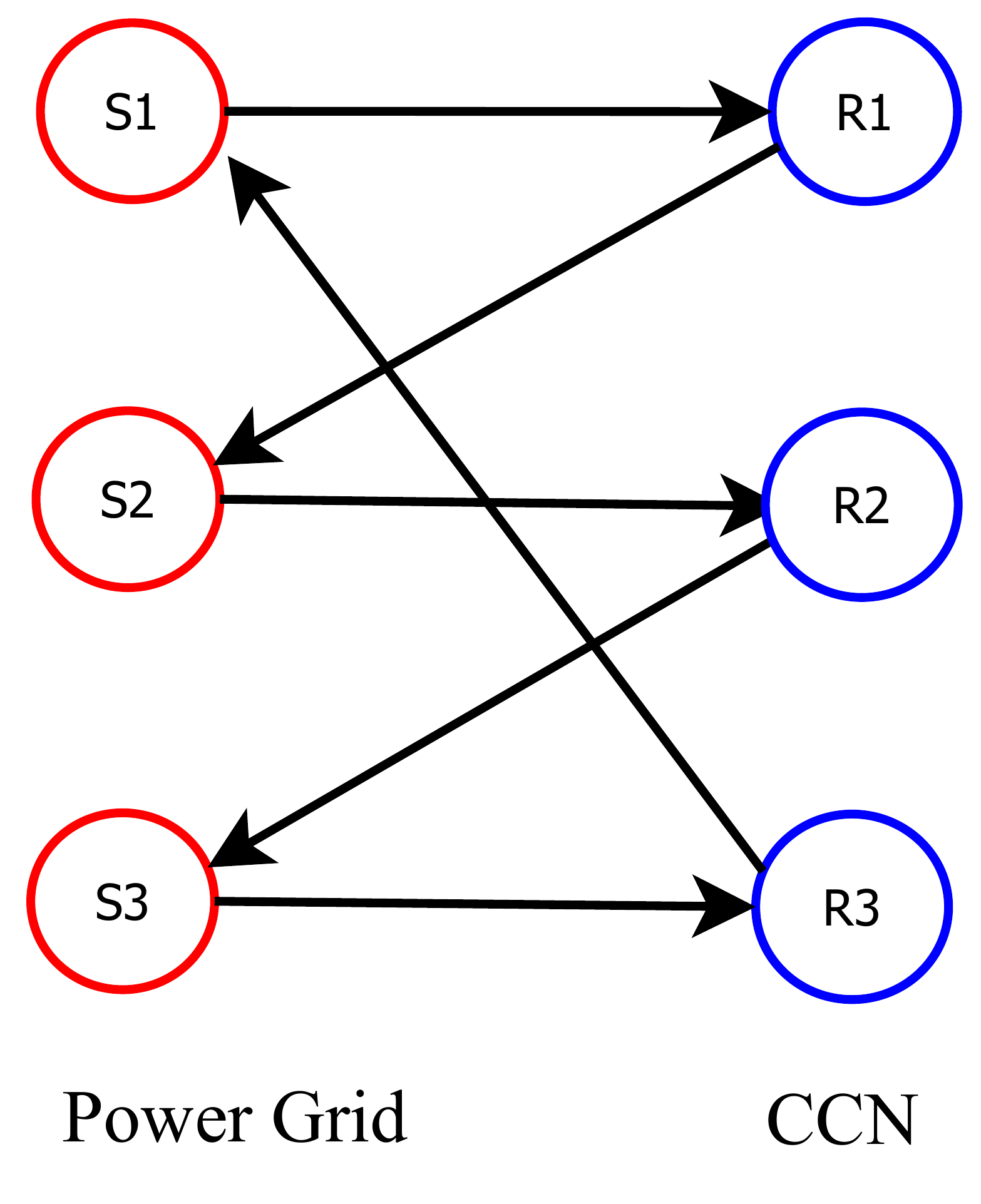}}                
\subfigure[Bidirectional Interdependency]
{\label{Bi1}\includegraphics[scale=0.25]{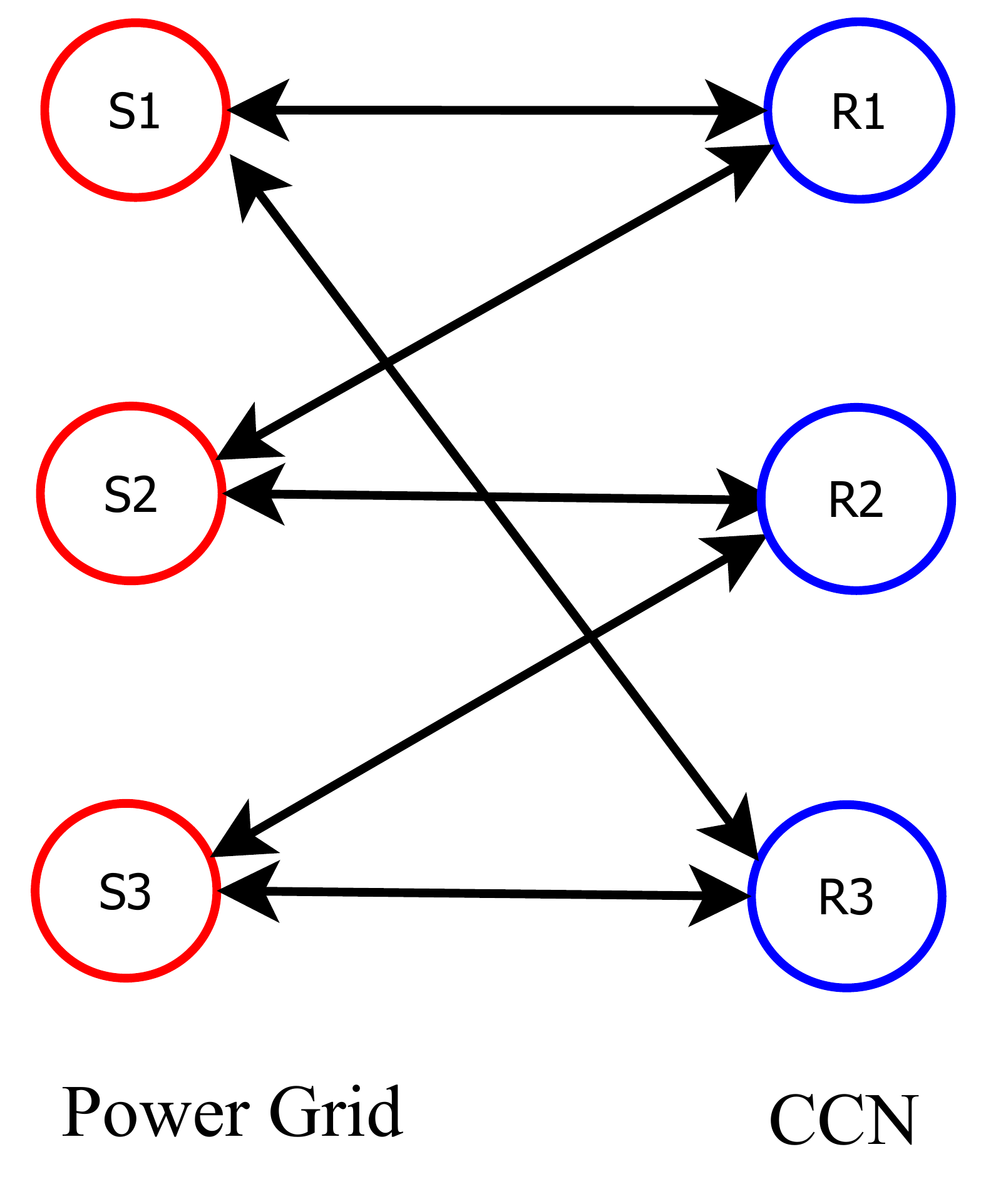}}
\caption{Graph structure under different interdependency models}
\label{Comparison}
\end{figure}

\begin{lemma}\label{lemma1}
A network with one or more operating nodes has at least one cycle.
\end{lemma}
\begin{proof}
We prove by contradiction that no node in a network can operate if there is no cycle in the network. Suppose that there is no cycle, i.e. all nodes are connected through one or more paths. First, remove all of the non-operating nodes; hence, the remaining nodes are operating, and the network is still acyclic. Now consider the starting node of one of the paths which is either a substation $S_i$ or a router $R_j$. This starting node does not have any incoming edges; therefore, substation $S_i$ (router $R_j$) does not receive any control (power) and cannot operate which is a contradiction with the assumption of nodes being operating. 

Now we show that existence of at least one cycle is sufficient to have an operating node. In a bipartite graph, every substation (router) in a cycle receives an incoming edge from a router (substation) in that cycle; thus, the nodes in that cycle can operate. If all the other nodes of the network receive an incoming edge directly or through a path starting from a node in that cycle, those nodes will be operating, too. 


\end{proof}

\begin{lemma}\label{lemma2}
To stop the operation of any cycle, one of the nodes in the cycle should be removed (Cycles are Stable Components).
\end{lemma}
\begin{proof}
By definition, every substation (router) in the bipartite graph remains operating if it has an incoming edge from a router (substation). Every node in a cycle receives at least one incoming edge from the nodes inside that cycle; therefore, the removal of nodes outside the cycle will not affect the operation of the nodes inside that cycle. As a result, to stop the operation of a cycle, one of its nodes must be removed. 
\end{proof}

Note that stopping the operation of a cycle is not equivalent to stopping the operation of all the nodes inside the cycle. If the nodes inside a cycle are isolated from the other nodes, removing exactly one node from the cycle will stop all of them from operating. However, if nodes inside a cycle receive incoming edges from other nodes outside of the cycle, more node removals are needed to cause the failure of all of the nodes in that cycle.

\begin{lemma}\label{lemma3}
For total failure, at least one node from every cycle should be removed .
\end{lemma}
\begin{proof}
By contradiction - Suppose that there exist a cycle so that none of its nodes are removed. By lemma \ref{lemma2}, all of the nodes inside that cycle remain operating, which contradicts the assumption of total failure.
\end{proof}

\begin{theorem}\label{HittingCycle}
The minimum number of nodes that hit all of the cycles in a bipartite graph is the optimal solution for the Node-MTFR problem.
\end{theorem}
\begin{proof}
Immediate from lemma \ref{lemma3}.
\end{proof}

\begin{corollary}
Finding the Node-MTFR in star networks with unidirectional interdependency is NP-complete.
\end{corollary}
\begin{proof}
By Theorem \ref{HittingCycle}, the Node-MTFR problem is a hitting cycle problem which is exactly equivalent to the well-known problem of Feedback Vertex Set (FVS). By definition, FVS in a graph finds the smallest set of nodes so that their removals make the graph acyclic; and it is known to be NP-complete for general graphs \cite{Garey}. Moreover, Cai \textit{et al.} proved that FVS is NP-complete for a special class of bipartite graphs called bipartite tournament \cite{Cai}. Therefore, the Node-MTFR problem which is finding FVS in general bipartite graphs is also NP-complete.
\end{proof}

\subsubsection{Problem Formulation}
It was shown in lemma \ref{lemma3} that finding the Node-MTFR is a hitting cycle problem. Next, we present a cycle-based Integer Linear Programming (ILP) formulation for this problem, assuming that all of the cycles are given. Let $N$ be a $n \times 1$ binary vector so that each component $N_j$ takes values $1$ if node $j$ is removed and $0$ otherwise. Let matrix $A \in R^{m \times n}$ be a mapping between the $m$ cycles and $n$ nodes, where $A_{ij}=1$ if cycle $i$ contains node $j$ and $A_{ij}=0$ otherwise. Let $e$ be a $m \times 1$ vector of ones. The Node-MTFR problem can be formulated as follows.

\begin{align}
\mbox{minimize} \quad &\sum_{j=1}^n N_{j}\label{MFRR-obj} \\
\mbox{subject to} \quad &A \times N \geq e  \label{hitting-cycle} \\
 \quad &N_{j} \in \{0,1\}, \quad j=1,\cdots,n \label{binary}
\end{align}

In this formulation, the objective is to minimize the number of node removals. Every row $i$ of constraint (\ref{hitting-cycle}) requires that at least one of the removed nodes should hit cycle $i$. In the following, we develop heuristics to solve the problem.

\subsubsection{Heuristics}\label{Ap_Alg}
Since computing the Node-MTFR is computationally difficult, we consider approximation algorithms that give a near-optimal set of node removals in polynomial time. As explained in section \ref{Unidir}, the node removals problem is equivalent to a hitting cycle problem. Thus, if we have the set of all cycles in the graph, we can apply a greedy algorithm devised for solving the hitting set problem. The input to the algorithm is the set of cycles (each cycle is defined as the set of nodes it contains), and the set of nodes in the graph. This cycle-based algorithm is an iterative algorithm that works as follows. In each iteration, it removes the node that is shared among maximum number of cycles, updates the set of cycles, and repeats until no cycle remains. 

This cycle-based algorithm needs the set of all cycles as input; however, in general, a graph may have an exponential number of cycles. To overcome this deficiency, we devise a new algorithm that relies on the degree of the nodes instead of the cycles. The input to the algorithm is the adjacency matrix of the graph. The algorithm is iterative: Each iteration starts with a pruning stage in which the algorithm removes all of the edges that do not belong to a cycle. In the next stage of the iteration, it removes the node that has the maximum outgoing degree. Next, the algorithm removes all nodes that fail as a result of the cascading effect of that removal. Finally, the algorithm updates the adjacency matrix of the graph and repeats the iteration until no node remains. 

In the following, we compare the performance of these algorithms with the optimal solution. We consider a random bipartite graph with $N$ nodes on each side. Since enumerating all the cycles requires exponential time, we keep the size of $N$ small, and limit the graph to have small cycles. To do that, instead of randomly generating edges, we randomly generate cycles of size 6 or smaller until all the nodes have at least one incoming edge. For each value of $N$, we generate 100 random graphs and then apply our algorithms to each graph in order to find the minimum node removals. Moreover, for the optimal solution, we solve the hitting cycle problem as given by (\ref{MFRR-obj})-(\ref{binary}) using CPLEX. As can be seen from Figure \ref{alg_compare}, on average, the degree-based algorithm gives a slightly larger number of nodes compared with the cycle-based algorithm and optimal solution; however, it is very fast as it does not need to enumerate all of the cycles.


\begin{figure}[ht]
\centering
\includegraphics[scale=0.5]{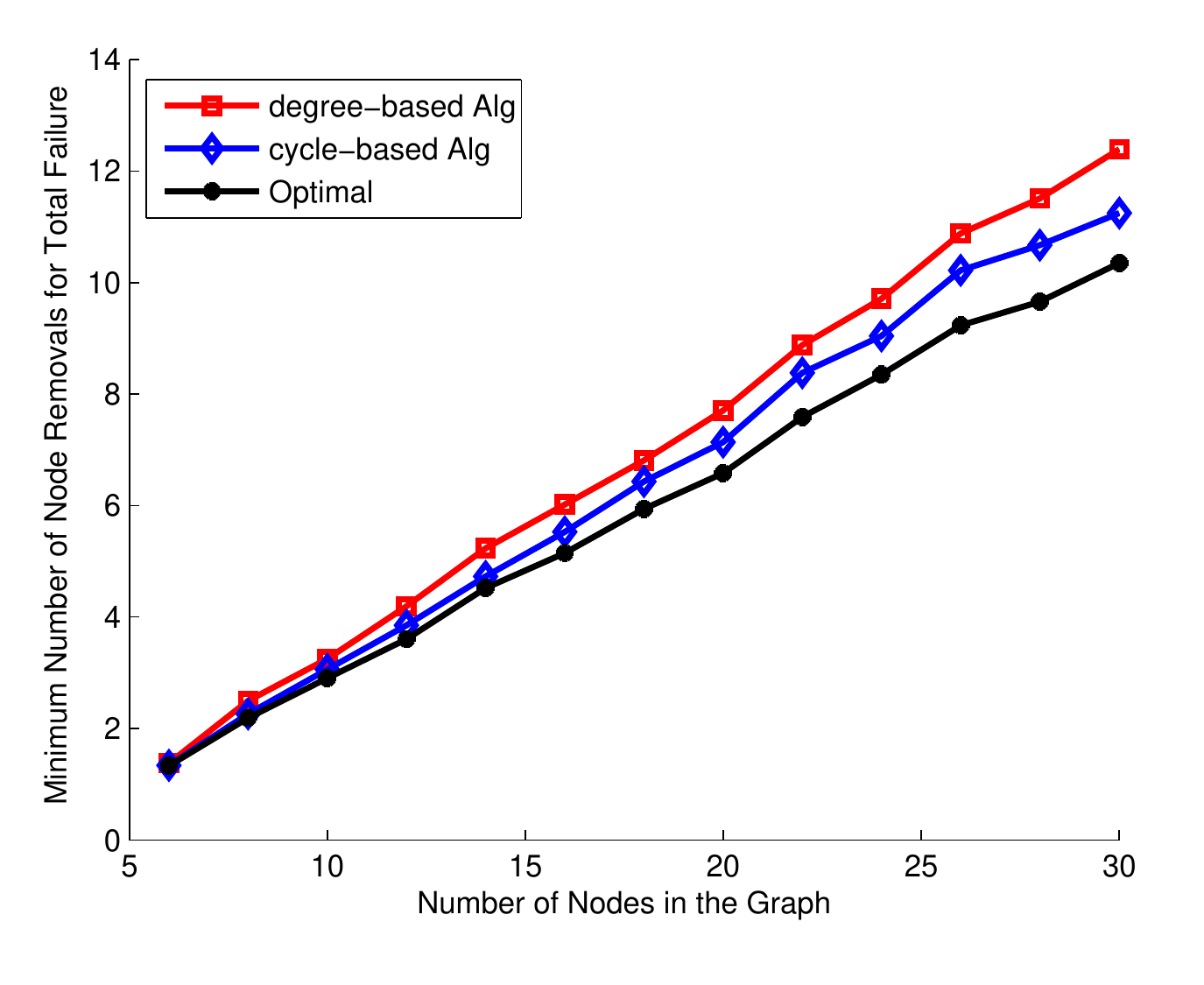}\vspace{-0.3cm}
\caption{Comparing different algorithms with optimal solution}
\label{alg_compare}\vspace{-0.2cm}
\end{figure}

\subsubsection{Minimum Edge Removals}
An alternative version of the problem is the minimum number of \textit{edges} needed to be removed to cause a total failure. Similar to lemmas \ref{lemma2} and \ref{lemma3}, to stop the operation of any cycle, one should remove one of its edges, and to have a total failure, at least one edge from every cycle should be removed. Consequently, we have the following results:

\begin{theorem}\label{HittingEdge}
The minimum number of edges that hit all of the cycles is the optimal solution for the Edge-MTFR problem.
\end{theorem}
\begin{proof}
Trivial.
\end{proof}

\begin{corollary}
Finding the minimum edge removals for total failure in the star networks with unidirectional interdependency is NP-complete.
\end{corollary}
\begin{proof}
By Theorem \ref{HittingEdge}, minimum edge removals problem is the edge version of the hitting cycle problem which is exactly equivalent to the well-known problem of Feedback Edge Set (FES). Similar to FVS, FES finds the smallest set of edges whose removals make the graph acyclic, and it is known to be NP-complete for general graphs \cite{Garey}. Furthermore, Guo \textit{et al.} proved that FES is NP-complete for bipartite tournaments \cite{Guo}. Since finding FES in bipartite tournaments is a special case of the Edge-MTFR problem, the Edge-MTFR problem is also NP-complete.
\end{proof}

\subsection{Bidirectional Interdependency}\label{Bidir}
Under the bidirectional interdependency model, the edges between the two interdependent networks are bidirectional, i.e. if a power node $S_i$ provides power to router $R_j$, router $R_j$ \textit{must} provide control to power node $S_i$ (Figure \ref{Bi1}).

\begin{theorem}\label{VertexCover}
Finding the minimum node removals for total failure in star networks with bidirectional interdependency is solvable in polynomial time.
\end{theorem}
\begin{proof}
It is easy to see from Figure \ref{Bi1} 
that edges are cycles of length two, and hitting cycles of length two guarantees hitting cycles of larger size. On the other hand, hitting at least one node in every cycle of size two is equivalent to finding the minimum vertex cover in bipartite graphs. By Konig's Theorem, finding the minimum vertex cover in bipartite graph is equivalent to maximum matching which is polynomially solvable \cite{Orlin}. Thus, finding the minimum node removals in star networks with bidirectional interdependency is polynomially solvable.
\end{proof}


For the Edge-MTFR, all of the edges must be removed. This is due to the fact that in the star networks with bidirectional interdependency every edge is a cycle, and cycles are robust components.

\subsection{Comparing the interdependency models}
We have seen that when networks have unidirectional interdependency, finding the optimal solution for the Node-MTFR problem is NP-complete; however, it can be solved in polynomial time when the networks have bidirectional interdependency. Here, we try to explain by way of an example why the analysis of unidirectional interdependency is more difficult than bidirectional interdependency. Figures \ref{Uni1} and \ref{Bi1} show two networks with the same topology under the different interdependency models. Suppose that in both networks, node $S_1$ is intentionally removed. It can be seen that the removal of $S_1$ in network \ref{Uni1} leads to the sequential failure of nodes $R_1$, $S_2$, $R_2$, $S_3$ and finally $R_3$. However, removal of $S_1$ in network \ref{Bi1} does not cause any failures, as all of the other nodes still have an incoming edge. Now, suppose that we want to cause the failure of node $R_2$ in network \ref{Bi1}. In this case, all the nodes attached to $R_2$ (nodes $S_1$ and $S_2$) must be removed. It can be seen that failure of $R_2$ cannot cause the failure of other nodes as it is already disconnected from the rest of the network. These observations indicate that in the case of unidirectional interdependency, a failure can cascade multiple times between the networks. However, in the case of bidirectional interdependency, a failure cascades only in one stage: either from the power grid to the CCN or from the CCN to the power grid. This makes the analysis of bidirectional interdependent networks more tractable.


Next we compare the robustness of the interdependency models. We use the random graphs generated in section \ref{Ap_Alg} to generate a new set of graphs with the same topology but bidirectional dependency. 
To compare the robustness of the two models, we find the optimal solution in the unidirectional graphs by solving the hitting set problem using CPLEX, and bidirectional graphs by solving the vertex cover problem. It can be seen from Figure \ref{dependency_compare} that for all values of $N$, networks with bidirectional dependency need more node removals; therefore, they are more robust to failures. This observation shows that the existence of more disjoint cycles and shorter cycles makes a network more robust.

\begin{figure}[ht]
\centering
\includegraphics[scale=0.5]{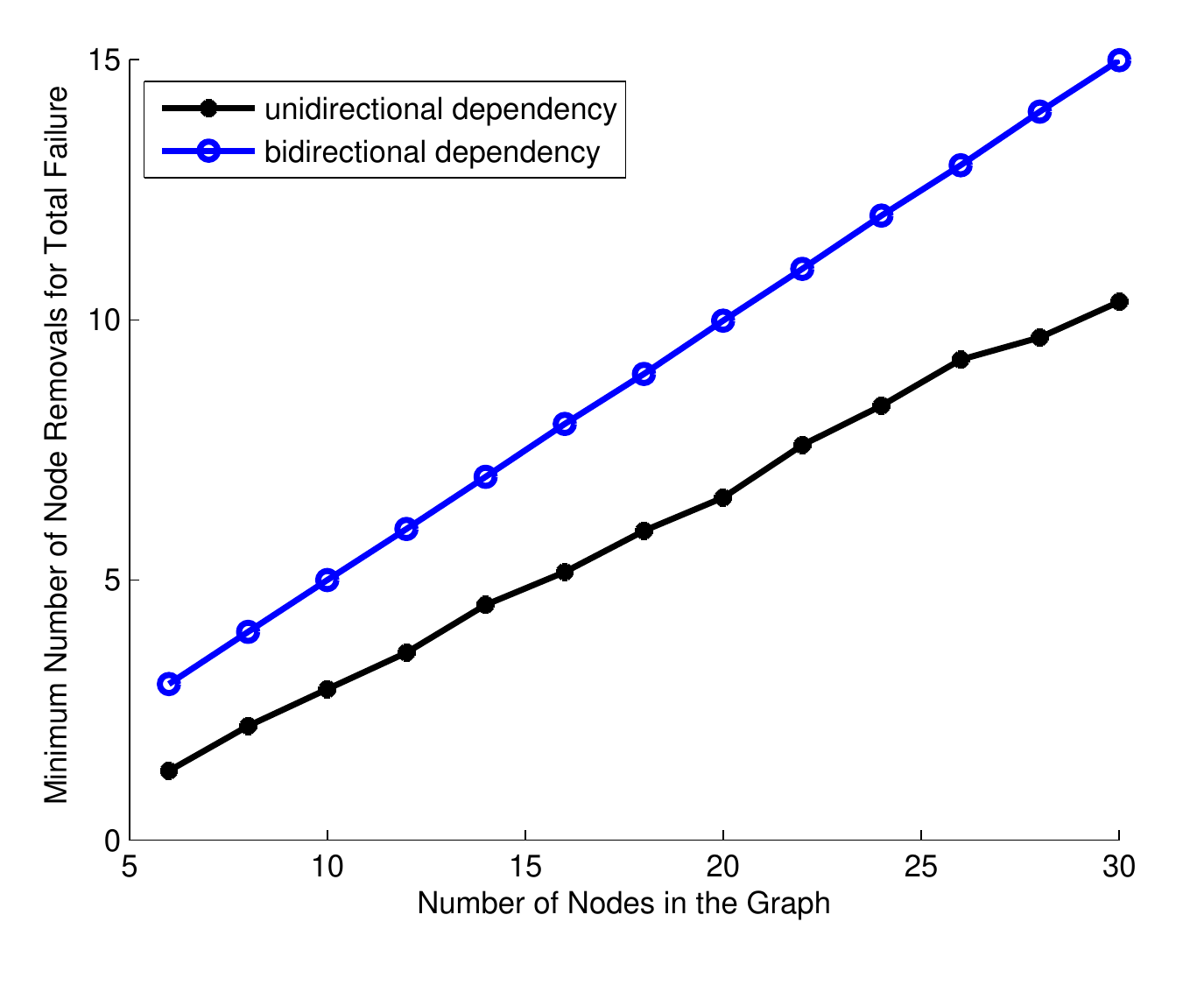}\vspace{-0.3cm}
\caption{Comparing the robustness of interdependency models}
\label{dependency_compare}\vspace{-0.2cm}
\end{figure}

\section{Discussion}\label{discussion}
We measured the Node-MTFR metric on the Italian network topology shown in \cite{Rosato}\footnote{The authors would like to thank professor Rosato for providing the data for the Italian power grid and communication networks.}. For the power grid, we only considered the substations that are directly connected to the generators. These substations are the most critical nodes as if they fail, all dependent substations will fail. Similarly, for the CCN, we only considered the routers that are directly connected to the control centers (we assumed that the control centers are located in the highly connected clusters of routers). Note that under this selection, the power grid and the CCN have star topologies. For the interdependency between the networks, we assumed that every substation receives its control signals from the nearest router, and every router receives power from the nearest substation. Figure \ref{Italy2} shows the power grid, the CCN and the interdependency between the two networks on the map of Italy.


Under the unidirectional interdependency model, we know that for a total failure, one should hit of all the cycles. It can be seen from Figure \ref{Italy2} that the cycles are very short, and mostly isolated. Therefore, one should remove a large number of nodes (13 nodes) simultaneously in order to cause the failure of the whole network. However, if we only consider the northwest of Italy, it can be seen that one third of the substations are controlled by a few routers. In fact, removing only three routers will lead to blackout in one third of the power grid. This shows that although the network might be robust to a total failure, only a few node removals can cause a considerable partial blackout. This observation leads to an important research direction: Measuring the effects of node removals on the failure of parts of the network.

\begin{figure}[ht]
\centering
\includegraphics[scale=0.20]{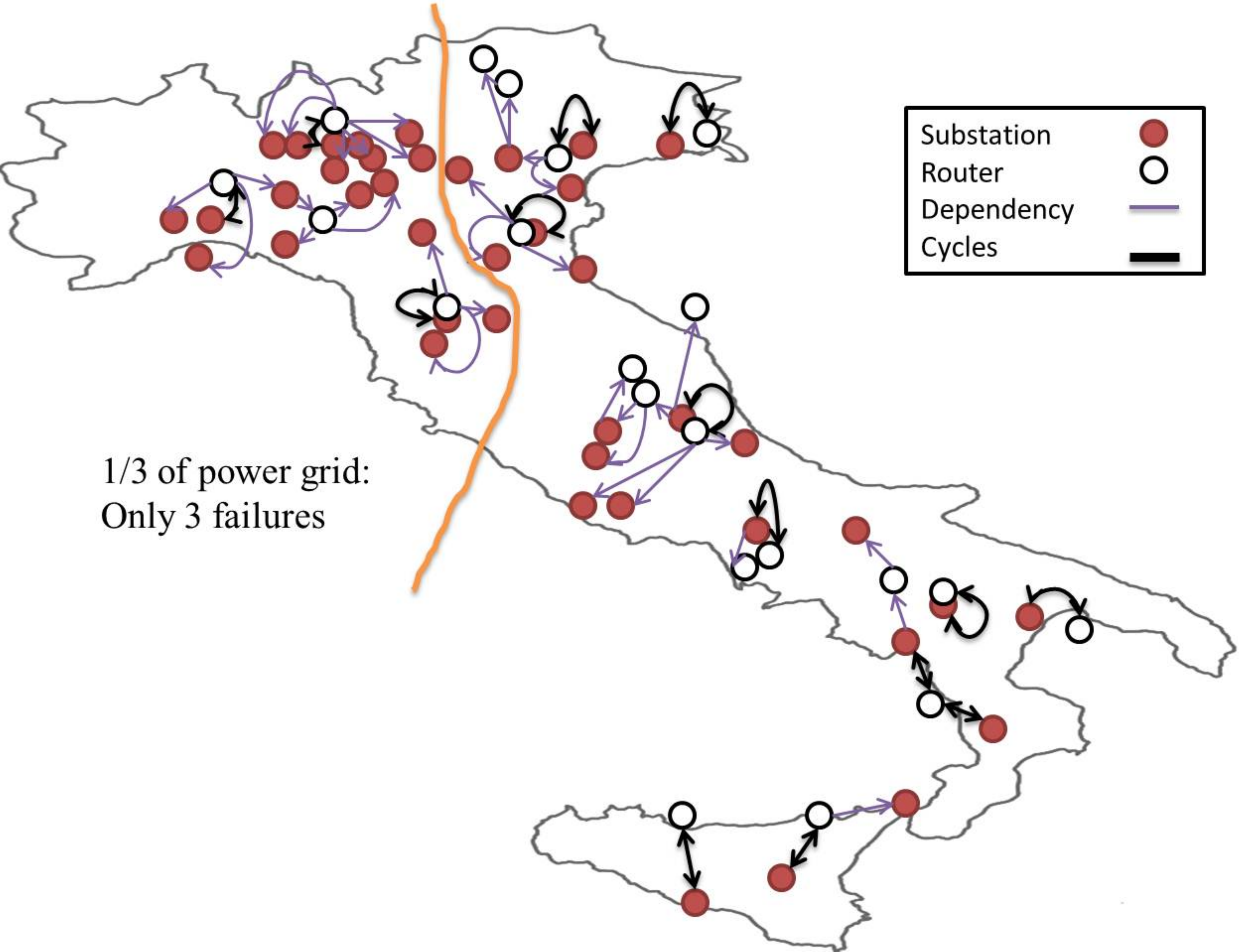}\vspace{-0.3cm}
\caption{Italian Power grid and CCN; Cycles shown in bold lines.}
\label{Italy2}\vspace{-0.2cm}
\end{figure}

\section{Conclusion}\label{conclusion}
We considered the problem of robustness in interdependent networks. Two networks A and B are said to be interdependent if the state of network A depends on the state of network B and vice versa. We considered a cyber-physical interdependency between the networks in the sense that the nodes in network A depend on information from network B, and the nodes in network B depend on the physical output of network A. The massive blackout in Italy in 2003 was the result of such interdependency between the power grid and the CCN where a small failure in the power grid cascaded between the two networks and led to extensive failures in both the power grid and the CCN. 

We studied the minimum number of nodes that should be removed from both networks so that all of the nodes in the networks fail after the ensuing cascades. We formulated this problem, and proved it is equivalent to the hitting cycles problem which is NP-complete in the case of unidirectional dependencies. We also presented polynomial algorithms which give suboptimal solutions. Moreover, for the case of bidirectional dependencies, we proved our problem is equivalent to the vertex cover problem in bipartite graphs which is solvable in polynomial time. We applied our results on the real network of Italy, and showed that it is robust to total failure. However, we showed that parts of the network are very vulnerable to failures.

\bibliographystyle{IEEEtran}
\bibliography{reference}

\end{document}